\documentclass[11pt]{amsart}
\usepackage{mathrsfs,latexsym,amsfonts,amssymb}
\newtheorem{theorem}{Theorem}[section]
\newtheorem{lemma}[theorem]{Lemma}
\newtheorem{corollary}[theorem]{Corollary}

\theoremstyle{remark}
\newtheorem{remark}[theorem]{Remark}
\theoremstyle{definition}
\newtheorem{definition}[theorem]{Definition}
\newtheorem{proposition}[theorem]{Proposition}
\newtheorem{example}[theorem]{Example}

\begin{document}

\title[The $G$-connected property and $G$-topological groups]
{The $G$-connected property and $G$-topological groups}

\author{Yongxing Wu}
\address{(Yongxing Wu): School of mathematics and statistics,
Minnan Normal University, Zhangzhou 363000, P. R. China}
\email{}

\author{Fucai Lin*}
\address{(Fucai Lin): School of mathematics and statistics,
Minnan Normal University, Zhangzhou 363000, P. R. China}
\email{linfucai@mnnu.edu.cn; linfucai2008@aliyun.com}

\thanks{The authors are supported by the NSFC (No. 11571158), the Natural Science Foundation of Fujian Province (No. 2017J01405) of China, the Program for New Century Excellent Talents in Fujian Province University, the Institute of Meteorological Big Data-Digital Fujian and Fujian Key Laboratory of Date Science and Statistics.}


\begin{abstract}
In this paper, we discuss some properties of of $G$-hull, $G$-kernel and $G$-connectedness, and extend some results of \cite{life34}. In particular, we prove that the $G$-connectedness are preserved by countable product. Moreover, we introduce the concept of $G$-topological group, and prove that a $G$-topological group is a $G$-topology under the assumption of the regular method preserving the subsequence.
\end{abstract}

\maketitle
\section{Introduction}
$G$-convergence is one of the important research objects in topology and analysis. On one hand, $G$-convergence is closely related to $G$-compactness, $G$-continuity and other related properties. On the other hand, it has played a fundamental role in mathematics and other fields applications. In addition to the general convergence of sequences, there are various
convergence type that are very important, both in pure mathematics and in other branches of science involving mathematics especially in information theory, biological science and dynamical systems. On the basis of several kinds of convergence properties in real analysis, Grosse-Erdmann and Connor \cite{life} discussed $G$-methods, which was defined on $G$-convergence on real spaces, a linear subspace of the vector space of all real sequences, and $G$-continuity for real functions, expounds the relationship among $G$-continuous functions, linear functions and continuous functions, generalized several known results and established the dichotomy theorem of $G$-continuity.

In 1946, the American Mathematical Monthly problem \cite{life1} put forward the idea of $G$-convergence, which is an extension of the idea of common convergence. Some authors, including Antoni \cite{life5}, Antoni and Salat \cite{life6}, Iwinski \cite{life3}, Posner \cite{life2}, Srinivasan \cite{life4} and Spigel and Krupnik \cite{life7} and so on, have studied $A$-continuity defined by a regular summability matrix $A$. Others authors, including Borsik and Salat \cite{life11}, \"{O}zt\"{u}rk \cite{life8}, Savas \cite{life9}, Savas and Das \cite{life10} and and so on, have studied $A$-continuity for methods of almost convergence or for related methods. In \cite{life12}, the authors introduced summability matrices. And in \cite{life13}, the authors discussed summability in topological groups. Ko\v{c}inac and Di Maio defined statistical convergence in topological spaces by \cite{life14}, discussed statistically Fr\'{e}chet spaces and statistically sequential spaces, and studied their applications in selection function spaces, hyperspaces, principles theory and so on.

In 1935, Zygmund first used the name ``almost convergence" to describe the concept of statistical convergence in his celebrated monograph published in Warsaw \cite{life15}. In 1951, Fast \cite{life16} and Steinhaus \cite{life17} were introduced respectively the notion of statistical convergence for real or complex number sequences. After that, it was reintroduced by Schoenberg \cite{life18}. And also, independently, by Buck \cite{life19}.

\c{C}akalli defined and studied the statistical convergence in topological groups in \cite{life20, life21, life22, life23}. In addition, \c{C}akalli \cite{life24} has also defined and studied $G$-sequential open and $G$-sequential closed subsets of a first countable Hausdorff topological group. And \c{C}akalli and Khan \cite{life25} defined and studied the statistical convergence in first countable Hausdorff topological spaces. In particular, by \cite{life14}, the authors has also introduced and investigated the notion of statistical convergence in topological spaces and uniform spaces, and show how this convergence can be extended to selection principles theory, function spaces, hyperspaces and so on. In \cite{life22}, \c{C}akalli has introduced the concept of $G$-compactness and  he has proved that the $G$-continuous image of any $G$-compact subset of X is also $G$-compact \cite[Theorem 7]{life22}. \c{C}akalli investigated $G$-continuity, and obtained further results in \cite{life26}. The reader can refer to
\cite{life27, life29, life30, life37, life28} for some other types of continuities which can not be given by any sequential method.

Recently, S. Lin and L. Liu in \cite{life31} discussed the concepts of $G$-method and $G$-convergence in topological spaces, and defined $G$-neighborhoods and studied the properties of $G$-continuity of mappings and so on. As is known to all, topological space can be described by some important sets, such as open sets, closed sets, neighbourhoods and derived sets, define convergence and depict continuity in topological spaces. Moreover, L. Liu gave some properties of $G$-neighborhoods, $G$-continuity at a point, $G$-derived sets and $G$-boundaries of a set in \cite{life32, life33}.

In this paper, the content is organized as follows. In this section 2, we introduce the necessary notation and terminology which are used for the rest of the paper. In section 3, we discuss some properties of $G$-hull and $G$-kernel and generalizes some results of \cite{life34}. In section 4, we study the $G$-connectedness and prove that the $G$-connectedness is preserved by countable product. In section 5, we introduce the concept of $G$-topological group, and prove that a $G$-topological group is a $G$-topology.

\smallskip
\section{Prelinminaries}
 Throughout this paper, $\mathbb{N}$ denotes the set of all positive integers and $X$ denotes a Hausdorff topological group with operations as defined in \cite[Definition 3.1]{life34}. We use boldface letters {\boldmath$x$}, {\boldmath$y$}, {\boldmath$z$}, $\cdots$ for sequences {\boldmath$x$}=($x_{n}$), {\boldmath$y$}=($y_{n}$), {\boldmath$z$}=($z_{n}$), $\cdots$ of terms of $X$. The sets $s(X)$ and $c(X)$ respectively denote all $X$-valued sequences and all $X$-valued convergent sequences of points in $X$. $G$ will be a regular method (see Definition~\ref{def1.1.2})
unless otherwise is stated.

\begin{definition}\label{def0.0.1}\cite[Definition 3.1]{life34}
Let {\boldmath$C$} be a category of groups with a set of operations $\Omega$ and with a set {\boldmath$E$} of identities such that {\boldmath$E$} includes the group laws, and the following conditions hold: If $\Omega_{i}$ is the set of $i$-ary operations in $\Omega$ for each $i=0, 1, 2$, then:

\smallskip
(1) $\Omega=\Omega_{0}\cup \Omega_{1}\cup \Omega_{2}$;

\smallskip
(2) The group operations written additively 0, $-$ and $+$ are the elements of $\Omega_{0}$, $\Omega_{1}$ and $\Omega_{2}$ respectively. Let $\Omega_{2}'=\Omega_{2}\setminus\{+\}$, $\Omega_{1}'=\Omega_{1}\setminus\{-\}$ and assume that if $\star \in \Omega_{2}'$, then $\star^{\circ}$ defined by $x \star^{\circ} y=y \star x$ is also in $\Omega_{2}'$. Also assume that $\Omega_{0}=\{0\}$;

\smallskip
(3) For each $\star\in \Omega_{2}'$, {\boldmath$E$} includes the identity $x\star(y+z)=x\star y+x\star z$;

\smallskip
(4) For each $\omega\in \Omega_{1}'$ and $\star\in \Omega_{2}'$, {\boldmath$E$} includes the identities $\omega(x+y)=\omega(x)+\omega(y)$ and $\omega(x)\star y=\omega(x\star y)$.

\smallskip
Then the category {\boldmath$C$} satisfying the conditions (1)$-$(4) is called {\it a category of groups with operations}.
\end{definition}

\begin{definition}\label{def0.0.2}\cite[Definition 3.4]{life34}
Let $X$ be a group with operations, i.e., an object of {\boldmath$C$}. A subset $A\subseteq X$ is called a {\it subgroup with operations} subject to the following conditions:

\smallskip
(1) $a\star b\in A$ for $a,b\in A$ and $\star \in \Omega_{2}$;

\smallskip
(2) $\omega(a)\in A$ for $a\in A$ and $\omega\in \Omega_{1}$.
\end{definition}

\begin{definition}\label{def0.0.4}\cite[Definition 3.6]{life34}
A category {\boldmath$Top^{C}$} of topological groups with a set $\Omega$ of continuous operations and with a set {\boldmath$E$} of identities such that {\boldmath$E$} includes the group laws such that the conditions (1)$-$(4) of Definition~\ref{def0.0.1} are satisfied, is called a {\it category of topological group with operations}.
A {\it morphism} between any two objects of {\boldmath$Top^{C}$} is a continuous group homomorphism, which preserves the operations in $\Omega_{1}'$ and $\Omega_{2}'$.
\end{definition}

The categories of topological groups, topological rings and topological $R$-modules are examples of categories of topological groups with operations.

 \begin{definition}\label{def1.1.2}
Let $X$ be a topological space.

\smallskip
(1) A method $G$ : $c_{G}(X) \to X$ is called {\it regular} \cite[Definition 1.1]{life31} if $c(X) \subset c_{G}(X$) and $G$({\boldmath$x$})=lim{\boldmath$x$} for each {\boldmath$x$} $\in$ $c$($X$), where lim denotes the limit function lim$x$ = lim$_{n\rightarrow\infty}$$x_{n}$ on $c(X)$.

\smallskip
(2) A method $G$ : $c_{G}(X) \to X$ is called {\it subsequential} \cite[Definition 1.1]{life31} if, whenever {\boldmath$x$} $\in c_{G}(X$) is $G$-convergent to $l\in X$, then there exists a subsequence {\boldmath$x^{'}$} $\in c(X$) of {\boldmath$x$} with lim{\boldmath$x^{'}$}=$l$.

\smallskip
(3) We say a method $G$ {\it preserves the $G$-convergence of subsequences} \cite[pp.1083]{life34} if, whenever a sequence {\boldmath$x$} is $G$-convergent with $G$({\boldmath$x$})=$l$, then any subsequence of {\boldmath$x$} is $G$-convergent the same point $l$.
\end{definition}

As a generalization of the concept of closures in topological spaces, $G$-hulls and $G$-closures are essential concepts in $G$-method, see \cite{life31}.

 \begin{definition}\label{def2.1}\cite[Definition 2.1]{life31}
Let $X$ be a set, and $G$ be a method on $X$. A subset $A$ of $X$ is called a {\it $G$-closed set} of $X$ if, whenever {\boldmath$x$} $\in s(A)\cap c_{G}(X$), then $G$({\boldmath$x$}) $\in A$. The set $A$ of $X$ is called {\it $G$-open} if $X\setminus A$ is $G$-closed in $X$.
\end{definition}

\begin{definition}\label{def2.2}
Let $X$ be s set, $G$ be a method on $X$ and $A$ $\subset$ $X$.

\smallskip
(1) The {\it $G$-hull} \cite[pp.102]{life} of $A$ is defined as the set $\{G$({\boldmath$x$}) : {\boldmath$x$} $\in s(A)\cap c_{G}(X)\}$, and the $G$-hull of $A$ is denoted by [$A$]$_{G}$.

\smallskip
(2) The {\it $G$-closures }\cite[Definition 2.4]{life31} of $A$ is defined as the intersection of all $G$-closed sets containing $A$, and the $G$-closures of $A$ is denoted by $\overline{A}$$^{G}$.

\smallskip
(3) We say that a subset of $X$ is {\it $G$-dense }\cite[pp.7]{life35} in $X$ if $\overline{A}^{G}=X$.

\smallskip
(4) The {\it $G$-kernel} of $A$ is defined as the set $\{l\in X$ : there is no {\boldmath$x$} $\in s(X\setminus A)\cap c_{G}(X$) with $l=G$({\boldmath$x$})$\}$, and the $G$-kernel of
$A$ is denoted by ($A)_{G}$.

\smallskip
(5) The {\it $G$-interior} of $A$ is defined as the union of all $G$-open sets contained in $A$, and the $G$-interior of $A$ is denoted by $A^{\circ G}$ is also $G$-open.

\smallskip
(6) A subset $A$ of $X$ is called a {\it $G$-neighborhood} of a point $x\in X$ if there exists a $G$-open set $U$ with $x \in U\subset A$.

\smallskip
(7) A subset $F$ of $A$ is called {\it $G$-closed} in $A$ if there exists a $G$-closed subset $K$ of $X$ such that $F=K\cap A$.
\end{definition}

\begin{remark}\label{rem2.3}
(1) Let $\widetilde{X}$ be the product $\prod_{\alpha\in I}X_{\alpha}$ of countable number of copies of $X_{\alpha}$. Denote a sequence in $\widetilde{X}$ by $\widetilde{\textbf{x}}=(x^{n})$ with $x^{n}=(x^{n}_{\alpha})_{\alpha\in I}$ for $n \in \mathbb{N}$. Then we have a method $\widetilde{G}$ on $\widetilde{X}$ defined by $\widetilde{G}(\widetilde{\textbf{x}})=G(\pi_{\alpha}(x^{n}))_{\alpha\in I}$. Hence $c_{\widetilde{G}}(\widetilde{X})$ has the sequences in $\widetilde{X}$ such that for each $\alpha\in I$ the sequence $(\pi_{\alpha}(x^{n}))$ is in $c_{G}(X)$, where $\pi_{\alpha}$ is the projection mapping to $X_{\alpha}$.

\smallskip
(2) Obviously, $A$ is $G$-closed if [$A]_{G}\subset A$. Moreover, if $G$ is regular method, then $A\subset[A]_{G}$, and hence $A$ is $G$-closed if and only if [$A]_{G}=A$.

\smallskip
(3) A subset $A$ of $X$ is called $G$-open if and only if $A$ $\subset (A)_{G}$.

\smallskip
(4) The empty set $\emptyset$ and the whole space $X$ are $G$-closed. It is clear that $\overline{\emptyset}^{G}=\emptyset$ and $\overline{X}^{G}=X$ for a regular method $G$. If $G$ is a regular method, then $A\subset \overline{A} \subset \overline{A}^{G}$. Even for a regular method, it is not always true that $\overline{(\overline{A}^{G})}^{G}=\overline{A}^{G}$.

\smallskip
(5) The union of any $G$-open subsets of $X$ is $G$-open. Here note that a subset $U$ of $A$ is $G$-open in $A$ if and only if there exists a $G$-open subset $V$ of $X$ such that $U=A \cap V$.
\end{remark}

\begin{definition}\label{def5.2}
Let $X$ be a set.

\smallskip
(1) A mapping $f:X \to X$ is called {\it $G$-continuous} \cite[pp.1080]{life34} if $G(f$({\boldmath{$x$}}))=$f(G$({\boldmath{$x$}})) for {\boldmath{$x$}}$\in c_{G}(X)$.

\smallskip
(2) A method $G$ is called {\it translate regular} if $G(g${\boldmath{$x$}})=$gG$({\boldmath{$x$}}) for {\boldmath{$x$}}$\in c_{G}(X)$ and $g \in X$, where $X$ is a group with operations.
\end{definition}

\begin{definition}\label{def4.1}\cite[Definition 4.1]{life34}
A non-empty subsets $A$ of $X$ is called {\it $G$-connected} if there are no non-empty disjoint $G$-closed subsets $F$ and $K$ of $A$ such that $A$=$F \bigcup K$. Particulary $X$ is called $G$-connected, if there are no non-empty, disjoint $G$-closed subsets of $X$ whose union is $X$.
\end{definition}

\smallskip
\section{Some properties of $G$-hull and $G$-kernel}
In this section, we discuss some properties of $G$-hull and $G$-kernel and generalizes some results of \cite{life34}. First, the following theorem give an improvement of \cite[Proposition 3.15]{life34}.

\begin{theorem}\label{the2.4}
Let $A_{t}$ be a subset of $X$ for each $t\in \Gamma$. Then we have the following statements.
 \begin{enumerate}
\item $[\prod_{t\in \Gamma}A_{t}]_{G}=\prod_{t\in \Gamma}[A_{t}]_{G}$.

\smallskip
\item If $A_{t}$ is $G$-closed for each $t\in \Gamma$, then $\prod_{t\in \Gamma}A_{t}$ is $G$-closed.
 \end{enumerate}
\end{theorem}

\begin{proof}
(1) Let $u=(u_{\alpha})_{\alpha\in I}$ in $[\prod_{\alpha\in I}A_{\alpha}]_{G}$. Then there is a sequence
{\boldmath{$\widetilde{x}$}}=$(x^{n})$  in \mdseries{$\prod_{\alpha\in I}A_{\alpha}$} such that $\widetilde{G}$({\boldmath$\widetilde{x}$})= $u$.
Hence G($\pi_{\alpha}(x^{n}))=u_{\alpha}$ for each $\alpha\in I$, and therefore $u\in \prod_{\alpha\in I}[A_{\alpha}]_{G}$. By the arbitrary of $u$, ones have [$\prod_{\alpha\in I}A_{\alpha}]_{G}\subset\prod_{\alpha\in I}[A_{\alpha}]_{G}$.

On the other hand, take an arbitrary ($a_{\alpha})_{\alpha\in I}$ in $\prod_{\alpha\in I}[A_{\alpha}]_{G}$. Then there are sequences $\pi_{\alpha}(x^{n})$ of the points $A_{\alpha}$  such that
G$(\pi_{\alpha}(x^{n}))=a_{\alpha}$ for each $\alpha\in I$. Since G$(\pi_{\alpha}(x^{n}))=a_{\alpha}$ for each $\alpha\in I$, it follows that ($a_{\alpha})_{\alpha\in I} \in [\prod_{\alpha\in I}A_{\alpha}]_{G}$. By the arbitrary of $(a_{\alpha})_{\alpha\in I}$, ones have $\prod_{\alpha\in I}[A_{\alpha}]_{G}\subset[\prod_{\alpha\in I}A_{\alpha}]_{G}$.

Therefore, we obtain that [$\prod_{t\in \Gamma}A_{t}]_{G}=\prod_{t\in \Gamma}[A_{t}]_{G}$.

(2) By (1), we have [$\prod_{t\in \Gamma}A_{t}]_{G}=\prod_{t\in \Gamma}[A_{t}]_{G}$. Since $A_{t}$ is $G$-closed for each $t\in \Gamma$, then
[$A_{t}]_{G} \subset A_{t}$ for each $t\in \Gamma$. So, $\prod_{t\in \Gamma}[A_{t}]_{G}\subset\prod_{t\in \Gamma}A_{t}$.
As a result we obtain that [$\prod_{t\in \Gamma}A_{t}]_{G}\subset\prod_{t\in \Gamma}A_{t}$. Therefore, $\prod_{t\in \Gamma}A_{t}$ is $G$-closed.
\end{proof}

In \cite[Theorem 3.16]{life34}, the authors proved the following result.

\begin{proposition}\label{pro3.3}\cite[Theorem 3.16]{life34}
If $G$ be a regular method preserving the $G$-convergence of subsequences and $A$ and $B$ are subsets of $X$, then ($A\times B)_{G}=(A)_G\times(B)_{G}$.
\end{proposition}

However, we can not extent this result to the case of infinity. Indeed, the next example is known as ($\prod_{n\in \mathbb{N}}A_{n})_{G} \ne\prod_{n\in \mathbb{N}}(A_{n})_{G}$.

\begin{example}\label{exam3.4}
There exists a regular method $G$ on $X$ preserving the $G$-convergence of subsequences such that there exists a sequence $\{A_{n}: n\in\mathbb{N}\}$ of subsets of $X$ with
($\prod_{n\in \mathbb{N}}A_{n})_{G}\ne\prod_{n\in \mathbb{N}}(A_{n})_{G}$.
\end{example}

\begin{proof}
Let $G$ be  the ordinary convergence method on the usual space $\mathbb{R}$. Then $G$ is a regular method preserving the $G$-convergence of subsequences. Let $X$=$\prod$$_{n=1}^\infty$$X_{n}$ and $c_{G}(X)=c(X)$, where each $X_{n}$=$\mathbb{R}$ and $X$ is endowed with the product topology.
We pick  $A_{n}$=($n-1/4$, $n+1/4$). So, ($A_{n}$)$_{G}$=($n-1/4$, $n+1/4$) for each $n\in \mathbb{N}$, then $$\prod_{n\in \mathbb{N}}(A_{n})_{G}=\prod_{n=1}^\infty(n-1/4, n+1/4).$$ We choose a sequence {\boldmath$y$} =(n)$_{n=1}^\infty$ of the points in $\prod_{n\in \mathbb{N}}(A_{n})_{G}$. For each $n\in\mathbb{N}$, let {\boldmath$x_{n}$}=($(x_{n})_{i})_{i=1}^\infty\in X$ such that for each $i\in\mathbb{N}$ ones have $$ {\boldmath {{(x_{n})}_{i}}}=
\begin{cases}
i& \text{i$\ne$n}\\
0& \text{i=n}
\end{cases}.$$
$\newline$
Since $\lim_{n\rightarrow\infty}(x_{n})_{i}=i$, we have $\lim_{n\rightarrow\infty}$({\boldmath$x_{n}$})=$(i)_{i=1}^\infty$={\boldmath$y$}.
Moreover, since $0\not\in A_{i}$ for each $i\in \mathbb{N}$, each {\boldmath$x_{n}$} $\not\in$ $\prod_{n\in \mathbb{N}}A_{n}$. Therefore, {\boldmath$y$} $\not\in(\prod_{n\in \mathbb{N}}A_{n})_{G}$.
As a result we obtain that $$(\prod_{n\in \mathbb{N}}A_{n})_{G}\ne\prod_{n\in \mathbb{N}}(A_{n})_{G}.$$
\end{proof}

The next result was proved in \cite{life34}. Now, we extend this result to the case of infinite.

\begin{proposition}\label{pro3.5}\cite[Theorem 3.19]{life34}
$\pi_{1}: X\times X\to X, (x, y)\mapsto x$ and $\pi_{2}: X\times X\to X$, $(x, y)\mapsto y$ projection mappings are $G$-continuous morphisms of topological groups with operations.
\end{proposition}

\begin{theorem}\label{the3.6}
For each $i\in \Gamma$, the mapping $\pi_{i}$: $\prod_{t\in \Gamma}X_{t}\to X_{i}$, ($x_{t})\mapsto x_{i}$ is a $G$-continuous morphism of topological groups with operations.
\end{theorem}

\begin{proof}
If {\boldmath$x$} is a sequence of the points of $\prod_{t\in \Gamma}X_{t}$ such that $G$({\boldmath$x$})=$u=(u_{t}$). Then, for each $i\in \Gamma$,
we have $G(\pi_{i}$({\boldmath$x$}))=G({\boldmath$x_{i}$})=$u_{i}=\pi_{i}(u)=\pi_{i}(G$({\boldmath $x$})).
\end{proof}

In \cite[Theorem 3.20]{life34}, the authors proved that each projection mapping is a $G$-open morphism of topological groups with operations. Indeed, each projection mapping is a $G$-closed morphism of topological groups with operations.

\begin{theorem}\label{the3.8}
Let $G$ be a regular method preserving the $G$-convergence of subsequences. Then $\pi_{1}:X\times X\to X$, ($x, y)\mapsto x$ and $\pi_{2}:X\times X\to X$, ($x, y)\mapsto y$ projection mappings are $G$-closed morphisms of topological groups with operations.
\end{theorem}

\begin{proof}
Let $A\subset X\times X$ be a $G$-closed subset. To prove that $\pi_{1}(A$) is $G$-closed. It suffices to prove [$\pi_{1}(A)]_{G}\subset\pi_{1}(A$). Let $A$ $\subset X\times X$ be a $G$-closed subset, $u\in [\pi_{1}(A)]_{G}$ and {\boldmath$x$} a sequence of the points in $X$ such that G({\boldmath$x$})=$u$. Choose a point $y\in X$, and let {\boldmath$y$}=($y, y, y,...)$ be the constant sequence. Since $A$ is $G$-closed, then the sequence({\boldmath$x$}, {\boldmath$y$}) of the points in $A$ such that G({\boldmath$x$}, {\boldmath$y$})=($u, y)\in A$. Then $\pi_{1}$(G({\boldmath$x$}, {\boldmath$y$}))=$G$({\boldmath$x$})=$u\in\pi_{1}(A$).
Then [$\pi_{1}(A)]_{G} \subset\pi_{1}(A$). As a result we obtain that $\pi_{1}$ is $G$-closed.

Similarly one can prove that $\pi_{2}$ is  also $G$-closed.
\end{proof}

In \cite{life34}, the authors also proved the following result. Then we also have the following Theorem~\ref{the3.10}.

\begin{proposition}\label{pro3.9}\cite[Theorem 3.23]{life34}
Let G be a method preserving the $G$-convergence of subsequences. Then for the projection mapping $\pi_{1}: X\times X\to X$, ($x, y)\mapsto x$ if $A\subset X$ is a $G$-open subset, then $\pi_{1}^{-1}(A$) is a $G$-open subset in $X\times X$.
\end{proposition}

\begin{theorem}\label{the3.10}
Let $G$ be a method preserving the $G$-convergence of subsequences. Then for the projection mapping $\pi_{1}: X\times X\to X$, ($x, y)\mapsto x$  if $A$ $\subset X$ is a $G$-closed subset, then $\pi_{1}^{-1}(A$) is a $G$-closed subset in $X\times X$.
\end{theorem}

\begin{proof}
Let $A\subset X$ be a $G$-closed subset. To prove that $\pi_{1}^{-1}(A$) is $G$-closed, it suffices to prove [$\pi_{1}^{-1}(A)]_{G} \subset \pi_{1}^{-1}(A$). Let ($u, v)\in [\pi_{1}^{-1}(A)]_{G}$ and ({\boldmath$x$}, {\boldmath$y$}) a sequence in $X\times X$ such that $G$({\boldmath$x$}, {\boldmath$y$})=($u, v$). Since $A$  is $G$-closed, G({\boldmath$x$})=$u \in A$. Then $\pi_{1}^{-1}(u) \in \pi_{1}^{-1}(A$), hence $\pi_{1}^{-1}(u)=\pi_{1}^{-1}(G$({\boldmath$x$}))=$G(\pi_{1}^{-1}$({\boldmath$x$}))=$G$({\boldmath$x$}, {\boldmath$y$})=($u, v$).
Therefore, ($u, v)\in\pi_{1}^{-1}(A$). As a result we obtain that $\pi_{1}^{-1}(A$) is a $G$-closed.
\end{proof}

\smallskip
\section{$G$-connected Topological Groups with Operations}
In this section, we prove that the $G$-connectedness is preserved by countably infinite product, which extend the following Theorem~\ref{t555}. First, we recall some results.

\begin{theorem}\cite[Theorem 4.8]{life34}\label{t555}
If $X$ is $G$-connected, then $X\times X$ is still $G$-connected.
\end{theorem}

First, we recall the following two results in \cite{life24}.

\begin{proposition}\label{pro4.2}\cite[Theorem 3]{life24}
Let $\{A_{i}: i\in \Gamma\}$ be a family of $G$-connected subsets of $X$. If $\bigcap_{i\in \Gamma}A_{i}$ is non-empty, then $\bigcup_{i\in \Gamma}A_{i}$ is $G$-connected.
\end{proposition}

\begin{proposition}\label{pro4.3}\cite[Theorem 1]{life24}
A $G$-continuous image of any $G$-connected subset of $X$ is $G$-connected.
\end{proposition}

In order to prove main result, we need the following result, which is also an improvement of Theorem~\ref{t555}.

\begin{proposition}\label{pro4.4}
If $X$ and $Y$ are $G$-connected,  then $X\times Y$ is $G$-connected.
\end{proposition}

\begin{proof}
Fix $a\in X$. If $X$ is $G$-connected, then $A=\{a\}\times Y$ is $G$-connected as the image of a
$G$-connected set under $G$-continuous map $f_{a} : X\longrightarrow X\times Y$; $x\mapsto(a, y)$. Similarly for each $y\in Y$ the subset
$B_{y}= X\times\{y\}$ is $G$-connected; and $A\cap B_{y}$ has a common point $(a, y)$, hence $A\cup B_{y}$ is $G$-connected by Proposition~\ref{pro4.2}. Since $X\times Y=\bigcup_{y\in Y}(A\cup B_{y})$. Thus it follows from Proposition~\ref{pro4.2} that $X\times Y$ is $G$-connected.
\end{proof}

Now, we can prove our main result.

\begin{theorem}\label{the4.6}
Let \{{$X$$_{n}$}\} $_{ n\in \mathbb{N}}$  be a sequence of topological spaces. Then the product space $X=\prod_{n\in \mathbb{N}}X_{n}$ is $G$-connected if and only if $X_{n}$ is $G$-connected for each $n\in \mathbb{N}$.
\end{theorem}

\begin{proof}
 Necessity. Let the product space $X=\prod_{n\in \mathbb{N}}X_{n}$ be $G$-connected. It follows from Theorem~\ref{the3.6} that $\pi_{i}: \prod_{n\in \mathbb{N}}X_{n}\to X_{i}$ is $G$-continuous morphism of topological groups with operations for each $i\in \mathbb{N}$. By Proposition~\ref{pro4.3}, each $X_{n}$ is $G$-connected.

Sufficiency. Suppose that $X_{n}$ is $G$-connected for each $n\in \mathbb{N}$. Fix an arbitrary point $a=(a_{n})_{n\in \mathbb{N}}\in\prod_{n\in \mathbb{N}}X_{n}$. We divide the proof into the following three steps.

\smallskip
{\bf Step 1:} We show that  for each $i\in \mathbb{N}$, the set $B_{i}=\prod_{n\in \mathbb{N}}A_{n}$ is $G$-connected, where $A_{n}=X_{n}$ if $n=i$ and $A_{n}=\{a_{n}\}$ if $n\ne i$.

\smallskip
Indeed, for each $i\in\mathbb{N}$, define the mapping $f_{i}: X_{i}\to B_{i}$ with $f(x)=(x_{n})_{n\in \mathbb{N}}$ for each $x\in X_{i}$, where for each $n\in\mathbb{N}$ ones have $$x_{n}=\begin{cases}
a_{n}& \text{n$\ne$i}\\
x& \text{n=i}
\end{cases}.$$

Since each $X_{i}$ is $G$-connected, it follows from Proposition~\ref{pro4.3} that it suffices to prove that each $f_{i}$ is $G$-continuous. Hence, it suffices to prove that $f_{i}(G$({\boldmath$x$}))=$G(f_{i}$({\boldmath$x$})) for each sequence {\boldmath$x$}=($x_{n})_{n\in \mathbb{N}}\subset X_{i}$ with G({\boldmath$x$})=$u\in X_{i}$.
Obviously, $f_{i}$(G({\boldmath$x$}))=$f_{i}(u)=(u_{n})_{n\in \mathbb{N}}$, where for each $n\in\mathbb{N}$ ones have $$u_{n}=\begin{cases}
a_{n}& \text{n$\ne$i}\\
u& \text{n=i}
\end{cases}.$$

On the other hand, G($f_{i}$({\boldmath$x$}))=G($f_{i}$($x$$_{n}$)$_{n\in \mathbb{N}}$). Let {\boldmath$y$}=($f_{i}(x_{n}))_{n\in \mathbb{N}}=({y}_{n})_{n\in \mathbb{N}}$, where, fore each $n\in\mathbb{N}$, $y_{n}=(y_{n}^{k})_{k\in\mathbb{N}}$ such that
$$y_{n}^{k}=\begin{cases}
a_{n}& \text{k$\ne$i}\\
x_{n}& \text{k=i}
\end{cases}.$$
Then G({\boldmath$y$})=G($f_{i}$({\boldmath$x$}))=($u_{n})_{n\in \mathbb{N}}$, where, for each $n\in\mathbb{N}$, $$u_{n}=\begin{cases}
a_{n}& \text{n$\ne$i}\\
u& \text{n=i}
\end{cases}.$$ Therefore, $f_{i}$(G({\boldmath$x$}))=G($f_{i}$({\boldmath$x$})). As a result we obtain that $f_{i}: X_{i}\to B_{i}$ is $G$-continuous. Therefore, each $B_{i}\subset\prod_{n\in \mathbb{N}}X_{n}$ is a $G$-connected subset.

\smallskip
{\bf Step 2:} Let $\mathcal{P}$ be the set of all finite subsets of $\mathbb{N}$. For each $S\in\mathcal{P}$, let $A_{n}=X_{n}$ for $n\in S$, and let $A_{n}=\{a_{n}\}$ for $n\in\mathbb{N}\setminus S$. Let $C_{S}=\prod_{n\in \mathbb{N}}A_{n}$ for each $S\in \mathcal{P}$, and let $C=\bigcap_{S\in \mathcal{P}}C_{S}$. We claim that $C$ is $G$-connected.

\smallskip
Indeed, since $a\in\bigcap_{S\in \mathcal{P}}C_{S}$, it follows from Proposition \ref{pro4.2} that it suffices to prove each $C_{S}$ is $G$-connected.
We define a mapping $p_{S}: Y=\prod_{n\in S}X_{n}\to C_{S}$ with $p_{S}(y)=({x}_{n})_{n\in \mathbb{N}}$ for each $y=(y_{n})_{n\in S}$, where, for each $n\in\mathbb{N}$, $$x_{n}=
\begin{cases}
a_{n}& \text{$n\in \mathbb{N}-S$}\\
y_{n}& \text{$n\in S$}
\end{cases}.$$ Similar to the proof of Step 1, it can prove that $p_{S}$ is a G-continuous mapping. Since each $X_{n}$ is a $G$-connected subset, $\prod$$_{n\in S}$$X_{n}$ is $G$-connected by Proposition \ref{pro4.4}. Therefore each $C$$_{S}$ is $G$-connected by the $G$-continuity of $p_{S}$.

\smallskip
{\bf Step 3:} It has $\overline{C}^{G}$=$X$.

\smallskip
Obviously, we have $\overline{C}^{G}\subset X$. Therefore, it suffices to prove that $X\subset\overline{C}^{G}$. Take an arbitrary $x\in X$. Let $x=(x_{n})_{n\in\mathbb{N}}$. and for each $i\in \mathbb{N}$. Then it is enough to find a sequence {\boldmath$y$} of the points in $C$ such that G({\boldmath$y$})=$x$. Indeed, let {\boldmath$y$}=(($y_{i}^{n}))_{n\in\mathbb{N}}$, where, for each $n, i\in \mathbb{N}$, $$y_{i}^{n}=\begin{cases}
x_{i}& \text{i $\le$ n}\\
a_{i}& \text{i $>$ n}
\end{cases}.$$
Then each $y_{i}^{n}\in C$, and hence {\boldmath$y$} is a sequence in $C$. Obviously, we have $G$({\boldmath$y$})=$x$, hence $x\in\overline{C}^{G}$. By the arbitrary choice of $x$, we have $X\subset\overline{C}^{G}$.

Therefore, it follows from \cite[Corollary 2]{life24} that $X$ is $G$-connected.
\end{proof}

\smallskip
\section{some properties of $G$-topological groups}
\smallskip
In this section, {\bf we assume that $G$ is a regular method preserving the subsequence.} We define the concept of $G$-topological group and obtain some properties of it. First, we recall some concepts and a fact.

\begin{definition}\label{def6.4}\cite[Definition 7.1]{life31}
Let $G_{1}, G_{2}$ be methods on sets $X$ and $Y$, respectively. A mapping $f:X\rightarrow Y$ is called {\it $(G_{1}, G_{2})$-continuous} if $f(x)\in C_{G_{2}}(Y)$ and $G_{2}(f(x))=f(G_{1}(x$))for each $x\in  C_{G_{1}}(X)$. $(G_{1}, G_{2})$-continuity is called {\it $G$-continuity} if $G_{1}$ and $G_{2}$ are the same method $G$.
\end{definition}

\begin{definition}\label{def6.3}\cite[Definition 6.1]{life31}
Let $G$ be a method on a set $X$. The family $\tau_{G}=\{A\subset X: A\ \mbox{is}\ G$-$\mbox{open}\}$ is called the {\it generalized topology} on the set $X$.

\smallskip
(1) $\tau_{G}$ is called a {\it $G$-topology} on the set $X$ if it is a topology on $X$.

\smallskip
(2) If $X$ carries a topology $\tau$ then ($X, \tau$)is called {\it $G$-topologizable} if $\tau=\tau_{G}$.
\end{definition}

\begin{proposition}\label{pro3.2}\cite[Theorem 3.12]{life34}
Let $G$ be a regular method preserving the $G$-convergence of subsequences, and $A$ a subset of $X$. Then the following are equivalent:

(1) $a\in (A)_{G}$.

(2) Any sequence {\boldmath$x$}=($x_{n}$) which is $G$-convergent to $a$ is almost in $A$.
\end{proposition}

\begin{definition}\label{def6.1}
Let $X$ be a group with operations. A topology $\tau$ on the set $X$ is a {\it $G$-topological group with operations} of $X$ provided that the following statements hold:

\smallskip
(1) The multiplication mapping $M: (X, \tau)\times(X, \tau)\rightarrow (X, \tau$) is $G$-continuous;

\smallskip
(2) The inverse mapping $In: (X, \tau)\rightarrow (X, \tau$) is $G$-continuous.

\smallskip
Let $X$ be a $G$-topological group with operations and $a \in X$. A family $\mathcal{B}_{a}$ of $G$-open neighbourhoods of $a$ is called a {\it fundamental system} \cite[pp.1087]{life34} of $G$-open neighbourhoods of $a$ if for each $G$-open neighbourhood $U$ of $a$, there is a $V \in \mathcal{B}_{a}$ such that $V \subset U$.
\end{definition}

First, we prove the following result which plays an important role in this section.

\begin{theorem}\label{the5.1}
A fundamental system $\mathcal{B}_{0}$ of $G$-open neighbourhoods of $0$ on a $G$-topological group $X$ with operations satisfies the following conditions:

(1) If $U \in \mathcal{B}_{0} $, then there exists $V \in  \mathcal{B}_{0}$ such that $V^{-1} \subset U$.

(2) The right translation and the left translation are all $G$-continuous.
\end{theorem}

\begin{proof}
(1) Let $U$ be a $G$-open neighbourhood of 0. Then it suffices to prove that the inverse mapping $f:X \to X$ is $G$-open. Hence it suffices to prove that $W^{-1}$ is $G$-open for each $W\in \mathcal{B}_{0} $. Take an arbitrary $W\in\mathcal{B}_{0}$. Then it suffices to prove $W^{-1} \subset (W^{-1})_{G}$. Indeed, pick any $x\in W^{-1}$ and a sequence {\boldmath{$x$}} being $G$({\boldmath{$x$}})$=x$. Next we shall prove that {\boldmath{$x$}} almost in $W^{-1}$ by Proposition~\ref{pro3.2}. Obviously, $x^{-1}\in W$ and $G((${\boldmath{$x$}})$^{-1})=G(f(${\boldmath{$x$}}$))=f(G$({\boldmath{$x$}}))$=f(x)=x^{-1}$ since $f$ is $G$-continuous. Since $W$ is $G$-open and $x^{-1}\in W$, it follows from $G((${\boldmath{$x$}})$^{-1})=x^{-1}$ that {\boldmath{$x$}}$^{-1}$ is almost in $W$ by Proposition~\ref{pro3.2}, hence {\boldmath{$x$}} is almost in $W^{-1}$. By the arbitrary of $x$, We have $W^{-1} \subset (W^{-1})_{G}$. So, $W^{-1}$ is a $G$-open neighbourhood of $0$.

Since $\mathcal{B}_{0}$ is a fundamental system of $G$-open neighbourhoods of 0, there is a $G$-open neighbourhood $V$ of 0 such that $V \subset U^{-}$. Hence $V^{-1}$ is a $G$-open neighbourhood of 0 and $V^{-1} \subset U$ as required.

(2) Fix an arbitrary $a\in X$. We prove that $f$ is $G$-continuous. Choose an arbitrary sequence {\boldmath{$x$}} of the point in $X$ such that $G$({\boldmath{$x$}})=$u$. Since $X$ is a $G$-topological group, it follows that $G(f$({\boldmath{$x$}}))=$G$({\boldmath{$x$}}$a$)=$G$($M$({\boldmath{$x$}}, {\boldmath{$a$}}))=$M$($G$({\boldmath{$x$}}, {\boldmath{$a$}}))=$M(G$({\boldmath{$x$}}), $G$({\boldmath{$a$}}))=$M(u, a)=ua$=$f$($G$({\boldmath{$x$}})), where ({\boldmath{$x$}}, {\boldmath{$a$}})$=\{(x_{n}, a_{n})\}_{n\in\mathbb{N}}$ and $a_{n}=a$ for each $n\in\mathbb{N}$. Thus, the right translation $f$ is $G$-continuous. Similarly one can prove that the left translation  is also $G$-continuous.
\end{proof}

\begin{corollary}
Let $X$ be a $G$-topological group with operations. Then the method $G$ is a translate regular.
\end{corollary}

Next we prove that $X$ is a $G$-topology if $X$ is a $G$-topological group.

\begin{theorem}\label{t5.7}
Let $\mathcal{B}_{e}$ be fundamental system of $e$ for a $G$-topological group $X$ with operations. Then, for each $x\in X$, the family $\mathcal{B}_{x}=\{xU:U \in \mathcal{B}_{e}\}$ is a  fundamental system of $x$.
\end{theorem}

\begin{proof}
First, we prove that $xU$ being a $G$-open neighborhood of $x$ for each $U\in\mathcal{B}_{e}$. Then it suffices to prove $xU\subset (xU)_{G}$. Take an arbitrary $y\in xU$ and a sequence {\boldmath{$z$}} such that $G$({\boldmath{$z$}})=$y$. Since $X$ is a $G$-topological group, $G$ is translate regular, hence $G(x^{-1}${\boldmath{$z$}})=$x^{-1}G$({\boldmath{$z$}})=$x^{-1}y\in U$. Since $U$ is $G$-open, there exists $N\in\mathbb{N}$ such that $x^{-1}z_{n}\in U$ for any $n>N$, thus $z_{n}\in xU$ for any $n>N$. Therefore, $xU$ is $G$-open.

Next, let $V$ be an arbitrary $G$-open neighborhood of $x$. We prove that there exists a $W\in\mathcal{B}_{x}$ such that $W\subset V$. By a similar proof above, it can obtain that $x^{-1}V$ is a $G$-open neighborhood of $e$, hence there exists a $V_{1}\in\mathcal{B}_{e}$ such that $V_{1}\subset x^{-1}V$. Put $W=xV_{1}$. Then $W=xV_{1}\subset x x^{-1}V=V$. The proof is complete.
\end{proof}

\begin{theorem}\label{t5.8}
Let $X$ be a $G$-topological group with operations. Then the family $\mathcal{B}=\bigcup_{x\in X}\mathcal{B}_{x}$ is a base for a topology $\sigma$ on $X$. Thus, $X$ is a $G$-topology.
\end{theorem}

\begin{proof}
Let $$\sigma=\{U\subset X:\ \mbox{for each}\ y\in U\ \mbox{there exists}\ V\in \mathcal{B}_{y}\ \mbox{such that}\ V\subset U\}.$$Then $\sigma$ is a topology on $X$. Obviously, it suffices to prove that $U\cap V\in\sigma$ for any $U, V\in\sigma$. Indeed, it suffices to prove $U\cap V$ being $G$-open, that is, $U\cap V\subset (U\cap V)_{G}$. Take an arbitrary $x\in U\cap V$ and a sequence {\boldmath{$z$}} such that $G$({\boldmath{$z$}})=$x$. Since $U$ and $V$ are $G$-open and $G$ preserving the subsequence, there exists $N\in\mathbb{N}$ such that $z_{n}\in U$ and $z_{n}\in V$ for each $n>N$, that is,  $z_{n}\in U\cap V$ for each $n>N$. Therefore, $U\cap V$ is $G$-open.
\end{proof}

Finally, we discuss the $G$-closure of a subset of $G$-topological group $X$.

\begin{corollary}\label{coro5.5}
Suppose that $X$ is a $G$-topological group with operations, and assume that $H$ is a $G$-topological subgroup with operations of $X$. If $H$ contains a non-empty $G$-open subset of $X$ then $H$ is $G$-open in $X$.
\end{corollary}

\begin{proof}
Let $U$ be a $G$-open non-empty subset of $X$ with $U \subset H$. Since $X$ is a $G$-topological group, it follows from Theorems~\ref{t5.7} and~\ref{t5.8} that for every $g \in H$ the set $\varrho_{g}(U)=Ug$ is $G$-open in $X$. Therefore, the set $H=\bigcup_{g \in H}Ug$ is $G$-open in $X$.
\end{proof}

By Theorem~\ref{the5.1}, we have the following corollary.

\begin{corollary}\label{coro5.6}
Let $f: X \to H$ be a homomorphism of $G$-topological groups with operations. If $f$ is $G$-continuous at the neutral element $e$ of $X$, then $f$ is $G$-continuous.
\end{corollary}

By \cite[Theorem 3.9]{life34}, we have the following proposition.

\begin{proposition}\label{the5.7}
Let $A$ be a subgroup with operations of a $G$-topological group $X$. If $A$ is $G$-open, then it is $G$-closed.
\end{proposition}

\begin{theorem}\label{the5.8}
Every $G$-topological group $X$ with operations has a $G$-open base at the identity consisting of symmetric neighbourhoods.
\end{theorem}

\begin{proof}
For an arbitrary $G$-open neighbourhood $U$ of the identity $e$ in $X$, let $V=U \bigcap U^{-1}$. By (1) of Theorem~\ref{the5.1}, the set $V$ is a $G$-open neighbourhood of $e$, and $V=V^{-1} \subset U$.
\end{proof}

\begin{lemma}\label{the5.9}\cite[Theorem 10]{life35}
Let $G$ be a regular method and $A\subseteq X$. If $x\in \overline{A}^{G}$, then for every $G$-open neighborhood $U$ of $x$, we have that $A\cap U\neq \emptyset$.
\end{lemma}

\begin{proposition}\label{the5.9.1}
Let $X$ be a $G$-topological group with operations and $A\subset X$. Then $x\in \overline{A}^{G}$ if and only if for each $G$-open neighborhood $U$ of $x$ ones have $A\cap U\neq \emptyset$.
\end{proposition}

\begin{proof}
By Lemma~\ref{the5.9}, it suffices to prove the sufficiency.

Assume that $x\not\in \overline{A}^{G}$, then there exists a $G$-closed set $F$ with $A\subset F$ such that $x\not\in F$. Therefore, $x$ belongs to $G$-open set $X\setminus F$. However, $(X\setminus F)\cap A=\emptyset$, which is a contradiction.
\end{proof}

\begin{proposition}\label{pro5.9}
Let $X$ be a $G$-topological group with operations, $U$ an $G$-open subset of $X$, and $A$ any subset of $X$. Then the set $AU$ (respectively, $UA$) is G-open in $X$.
\end{proposition}

\begin{proof}
Since $AU=\bigcup_{g \in A}\lambda_{g}(U)$ (respectively, $UA=\bigcup_{g \in A}\varrho_{g}(U))$, the conclusion follows.
\end{proof}

\begin{proposition}\label{pro5.10}
Let $X$ be a $G$-topological group with operations. Then, for every subset $A$ of $X$ and every $G$-open neighbourhood $U$ of the neutral element $e$, $\overline{A}^{G} \subset AU$.
\end{proposition}

\begin{proof}
Since $X$ is a $G$-topological group, there exists a G-open neighbourhood $V$ of the neutral element $e$ such that  $V^{-1} \subset U$. Take any $x \in \overline{A}^{G}$, then $xV$ is a $G$-open neighbourhood of $x$ by Theorems~\ref{t5.7} and~\ref{t5.8}. Therefore, it follows from Lemma~\ref{the5.9} that there is $a \in A \bigcap xV$, that is, $a=xb$, for some $b \in V$. Then $x=ab^{-1} \in AV^{-1} \subset AU$, hence, $\overline{A}^{G} \subset AU$.
\end{proof}

\begin{theorem}\label{the5.11}
Let $X$ be a $G$-topological group with operations, and $\mathcal{B}_{e}$ a base of the space $X$ at the neutral element $e$. Then, for every $A$ of $X$, $\overline{A}^{G}=\bigcap\{AU:U\in\mathcal{B}_{e}\}$.
\end{theorem}

\begin{proof}
By Proposition \ref{pro5.10}, ones have $\overline{A}^{G} \subset \bigcap\{AU:U\in\mathcal{B}_{e}\}$. It suffices to prove $\bigcap\{AU:U\in\mathcal{B}_{e}\} \subset \overline{A}^{G}$. Let $x \notin \overline{A}^{G}$, then there exists a G-open neighbourhood $W$ of $e$ such that ($xW) \bigcap A=\emptyset$. Take $U$ in $\mathcal{B}_{e}$ satisfying the condition $U^{-1} \subset W$. Then ($xU^{-1}) \bigcap A=\emptyset$, which obviously implies that $AU$ does not contain $x$. (Otherwise, assume $x \in AU$, then ($xU^{-1}) \bigcap A \ne \emptyset$, which is a contradiction). Thus, $\bigcap\{AU:U\in\mathcal{B}_{e}\} \subset \overline{A}^{G}$. Hence, for every $A$ of $X$, $\overline{A}^{G}=\bigcap\{AU:U\in\mathcal{B}_{e}\}$.

Similarly, the equality $\overline{A}^{G}=\bigcap\{UA:U\in\mathcal{B}_{e}\}$ holds for $G$-topological group $X$.
\end{proof}

\begin{proposition}\label{pro5.12}
Let $X$ be a $G$-topological group with operations. Then, for any symmetric subset $A$ of $X$, then $\overline{A}^{G}$ in $X$ is also symmetric.
\end{proposition}

\begin{proof}
Take an arbitrary $x\in\overline{A}^{G}$.It suffices to prove $x^{-1}\in\overline{A}^{G}$. Then it follows from Proposition~\ref{the5.9.1} that it suffices to prove that $U\cap A\neq\emptyset$ for each $G$-open neighborhood $U$ of $x^{-1}$. Let $U$ be an arbitrary $G$-open neighborhood of $x^{-1}$. Since $X$ is a $G$-topological group, it follows from Theorem~\ref{the5.1} that $U^{-1}$ is a $G$-open neighborhood of $x$. Since $x \in \overline{A}^{G}$, it follows from Proposition~\ref{the5.9.1} that $U^{-1}\cap A\neq\emptyset$, thus $U\cap A^{-1}=U\cap A\neq\emptyset$. Therefore, $x^{-1}\in\overline{A}^{G}$.
\end{proof}

{\bf Acknowledgements}. The authors are thankful to the
referee for valuable remarks and corrections and all other sort of help related to the content of this article.


\end{document}